\documentclass[12pt,reqno]{amsart}
 \usepackage[dvips]{epsfig}
 \usepackage{amsgen, amstext,amsbsy,amsopn, amsthm, amsfonts,amssymb,amscd,amsmat
 h,euscript,enumerate,url,verbatim,calc,xypic}
 \usepackage{hyperref}

 \usepackage{MnSymbol}

 \usepackage{latexsym}
 \usepackage{graphics}
 \usepackage{color}
 \usepackage{hyperref}
\usepackage{wrapfig}

 \newcommand{\m}{\mathfrak{m} }

  \newcommand{\Ass}{\operatorname{Ass}}

  \newcommand{\reg}{\operatorname{reg}}

\theoremstyle{plain}
 \newtheorem{theorem}{Theorem}[section]

 \newtheorem{lem}[theorem]{Lemma}
 \newtheorem{proposition}[theorem]{Proposition}

 \newtheorem{thm}[theorem]{Theorem}
 \newtheorem{cor}[theorem]{Corollary}

\newtheorem{observation}[theorem]{Observation}
 \theoremstyle{definition}
 
 \newtheorem{remark}[theorem]{Remark}
 \newtheorem{defn}[theorem]{Definition}
 
 \newtheorem{example}[theorem]{Example}
 \theoremstyle{remark}

\title[Regularity of symbolic powers of certain graphs] {Regularity of symbolic powers of certain graphs}
\author[B. Chakraborty and M. Mandal ]{Bidwan Chakraborty$^\dag$ and Mousumi Mandal$^*$}
\thanks{$^\dag$ Supported by CSIR grant No.: $09/081(1303)/2017\mbox{-EMR-I}$, India}
 \thanks{$^*$ Supported by SERB(DST) grant No.: $\mbox{EMR}/2016/006997$, India}

\address{Department of Mathematics, Indian Institute of Technology Kharagpur, 721302, India} \email{bidwan@iitkgp.ac.in}
 \address{Department of Mathematics, Indian Institute of Technology Kharagpur, 721302, India} \email{mousumi@maths.iitkgp.ac.in}

\begin{document}
 \maketitle
 \begin{abstract}
Let $G_{n,r}$ denote the graph with $n$ vertices $\{x_1,\ldots,x_n\}$ in cyclic order and for each vertex $x_i$ consider the set $A_i=\{x_{i-r},\ldots,x_{i-1},x_{i+1},x_{i+2},\ldots, x_{i+r}\},$ where $x_{i-j}$ is the vertex $x_{n+i-j}$, whenever $i<j$ and $0\leq r\leq \Bigl\lfloor\dfrac{n}{2}\Bigr\rfloor -1$. In $G_{n,r}$,  every vertex $x_i$ is adjacent to all the vertices of $V(G_{n,r})\backslash  A_i$. Let $I=I(G_{n,r})$ be the edge ideal of $G_{n,r}$. We show that Minh's conjecture is true for $I,$ i.e. regularity of ordinary powers and symbolic powers of $I$ are equal. We compute the Waldschmidt constant and resurgence for the whole class.
 \end{abstract}
 \section{Introduction}

Let $R = k[x_1,\ldots,x_n ]$ be a polynomial ring over a field $k.$ Let $M$ be a finitely generated graded $R$-module. The Castelnuovo-Mumford regularity (or simply, regularity) of $M,$ denoted by $\reg(M),$ is defined as $\reg(M) = \max\{a_i(M) + i| i \geq 0\},$ where $a_i(M)$ denotes the largest non-vanishing degree of the $i$-th local cohomology module of $M$ with respect to the graded maximal ideal of $R.$ It is an important invariant in commutative algebra and algebraic geometry.
There is a  correspondence between quadratic square free monomial ideals of
$R$ and finite simple graphs with $n$ vertices corresponding to the variables of $R$. To every finite simple graph $G$ with vertex set
$V(G)= \{x_1 ,\ldots,x_n\}$ and edge set $E(G),$ we associate its edge ideal $I=I(G)$ defined
by $I(G) =\{ x_ix_j \mid x_ix_j \in E(G)\}\subseteq R.$ For $n\geq 1$, the $n$-th symbolic power of $I$ is defined as $I^{(n)}=\displaystyle{\bigcap_{p\in \Ass I}(I^nR_p\cap R)}.$ In general computing the generators of symbolic powers of an ideal is a very difficult job so as to compute the regularity of the symbolic powers of ideals. Apart from regularity, it is interesting to study for which values of $r$ and $m$ the containment $I^{(r)}\subseteq I^m$ holds. To answer this question C. Bocci and B. Harbourne in \cite{bocci2010} defined an asymptotic quantity known as resurgence which is defined as  $\rho{(I)} =\sup \tiny\{ \frac{s}{t} ~~|~~I^{(s)}\nsubseteq I^t \tiny\}$  and showed that it exists for radical ideals. Since computing the exact value of resurgence is difficult, another asymptotic invariant $ \widehat{\alpha}{(I)} = \displaystyle{\lim_{s\rightarrow\infty}{} \frac{\alpha{(I^{(s)})}}{s}}~,$  known as Waldschmidt constant was introduced by Waldschmidt in \cite{waldschmidt1977proprietes}, where $\alpha(I)$ denotes the least generating degree of $I$.

This paper is mainly motivated by the conjecture of N. C. Minh which predict that if $G$ is a finite simple graph, then
$\reg I(G)^{(s)}=\reg I(G)^s $ for all $s \geq 1.$ By the result of Simis, Vasconcelos and Villarreal in \cite{svv}, the conjecture is true for bipartite graphs as  $I(G)^{(s)} = I(G)^s$ for all $s\geq 1$. Thus it is interesting to study the conjecture for non-bipartite graphs. In this direction in \cite{gu2018symbolic}, Gu et al. have proved the conjecture for odd cycles. Recently many researchers are working in this direction and in \cite{jayanthan}, Jayanthan and Kumar have proved the conjecture for certain class of unicyclic graph, in  \cite{chakraborty2019invariants} we proved it for complete graph and in \cite{fakhari},\cite{fakharichordal}, Seyed Fakhari has solved the conjecture for unicyclic graph and chordal graph respectively. In \cite{dipasquale2019asymptotic}, DiPasquale et al. have studied resurgence and asymptotic resurgence.

In this paper, we investigate the invariants for the class of graph $G_{n,r}$ described in the abstract. One description of such graph is given in section $2$ in figure \ref{overflow}.
 We see that for different values of $r$, $G_{n,r}$ covers a large class of graphs. For example if $r=0$ and $r=1$, $G_{n,r}$ gives the class of complete graphs and anticycles respectively. If $r=\Bigl\lfloor\dfrac{n}{2}\Bigr\rfloor -1$, then $G_{n,r}$ become bipartite graph or odd cycles depending on $n$ is even or odd. Here we describe a unified way to compare the regularity of the symbolic powers and ordinary powers of edge ideals of $G_{n,r}$ for different values of $n$ and $r$. Instead of studying separately for the individual classes of graphs we show that Minh's conjecture is true for the whole class of $G_{n,r}$. In section 2, we define the terminologies and recall the basic results to be used in the rest of the paper. In section 3, we describe the condition when $G_{n,r}$ become unmixed and in Lemma \ref{unmixed}, we show that for $r=1,2$, $G_{n,r}$ is always unmixed. In proposition \ref{gapfree}, we show that for $r=1$ and $2$, $G_{n,r}$ is gap-free graph. In section 4, we give an explicit description of the elements of the symbolic powers of edge ideals of $G_{n,r}$ when it is unmixed. Using the local cohomology technique in Theorem \ref{reginequality}, we show that Minh's conjecture is true for the  graph $G_{n,r}$ for all values of $n$ and $r$. In section 5, we compute the invariants like Waldschmidt constant and resurgence for the edge ideal of $G_{n,r}$.

 \section{Preliminaries}
 In this section, we collect the notations and terminologies used in the rest of the paper.
 \begin{defn}
   Let $V' \subseteq V(G)=\tiny\{x_1,x_2,\ldots,x_n\tiny\}$ be a set of vertices. For a monomial $ x^{\underline{a}} \in k[x_1,\ldots,x_n]$ with exponent vector $\underline{a} = (a_1,a_2,\ldots ,a_n)$ define the vertex weight $W_{V'}{(x^{\underline{a}})}$ to be $$ W_{V'}{(x^{\underline{a}})} := \sum_{x_i\in V'}{a_i}. $$
 \end{defn}
\begin{defn}
The vertex independence number of a graph is the cardinality of the largest independent vertex set and it is denoted by $\alpha(G).$ Formally,
 $$\alpha(G)=\max\{|U| : U \subseteq V(G)\text{ is independent set}\}.$$

\end{defn}

Now we recall some results which  describe the symbolic powers of the edge ideal in terms of minimal vertex covers of the graph.

\begin{lem}\cite[Corollary 3.35]{vantuylbook}
Let $G$ be a graph on vertices $\{x_1,\ldots,x_n\}, I=I(G)\subseteq k[x_1,\ldots,x_n]$ be the edge ideal of $G$ and $V_1,\ldots,V_r$ be the minimal vertex covers of $G.$ Let $P_j$ be the monomial prime ideal generated by the variables in $V_j.$ Then $$I=P_1 \cap\cdots\cap P_r$$ and $$I^{(m)}=P_1^{m} \cap\cdots\cap P_r^{m}.$$
\end{lem}

\begin{lem}\cite[Lemma 2.6]{bocci2016}\label{psymbolic}
Let $I \subseteq R$ be a square free monomial ideal with minimal primary decomposition $I=P_1\cap\cdots \cap P_r ~~with ~~P_j = (x_{j_1},\dots ,x_{j_{s_j}})$ for $j= 1,\ldots ,r$. Then $ {x_1^{a_1}}\cdots {x_n^{a_n}} \in I^{(m)} \mbox{ if and only if } a_{j_1}+\dots +a_{j_{s_j}}\geq m$ for $j=1,\dots,r$.

\end{lem}

 Using Lemma \ref{psymbolic}, and the concept of vertex weight Janssen et al. in  \cite{janssen2017comparing} described  the elements of symbolic powers of edge ideals as follows
$$ I^{(t)} = (\{ x^{\underline{a}} ~|\mbox{ for all minimal vertex covers }V^{\prime}, W_{V^{\prime}}(x^{\underline{a}})\geq t \}).$$
Further they have divided the elements of the symbolic powers of edge ideals into two sets written as  $I^{(t)}=(L(t))+(D(t)),$ where
  $$L(t) =  \{ x^{\underline{a}}~ | \deg(x^{\underline{a}})\geq 2t\mbox{ and for all minimal vertex covers } V^{\prime}, W_{V^{\prime}}(x^{\underline{a}})\geq t \}$$ and
$$D(t) = \{ x^{\underline{a}}~ | \deg(x^{\underline{a}})< 2t\mbox{ and for all minimal vertex covers } V^{\prime}, W_{V^{\prime}}(x^{\underline{a}})\geq t \}.$$
Thus for any graph, if we are able to identify the elements in $L(t)$ and $D(t)$ then we will be able to describe $I^{(t)}$.

%

 \begin{defn}
Let $ x^{\underline{a}} \in k[x_1,\ldots,x_{n}]$ be a monomial and $G$ be a finite simple connected graph on the set of vertices $\{x_1,\dots,x_n\}$. Let $\{e_1,e_2,\dots,e_r\}$ denote the set of edges in the graph. We may write $ x^{\underline{a}}=x_1^{a_1}x_2^{a_2} \cdots x_{n}^{a_{n}}e_1^{b_1}e_2^{b_2}\cdots e_{r}^{b_{r}} $,
where $b(x^{\underline{a}}) := \sum{b_j}$ is as large as possible.  When $x^{\underline{a}}$ is written in this way, we call this an optimal form of $x^{\underline{a}}$ or we say that $x^{\underline{a}}$ is expressed in optimal form, or simply $x^{\underline{a}}$ is in optimal form. In addition, each $ x_i^{a_i} $ with $a_i >0 $ in this form is called an ancillary factor of the optimal form, or just ancillary for short and $x_i$'s are called ancillary vertices.

\end{defn}
Note that optimal form of a monomial need not be unique but $b(x^{\underline{a}})$ is unique.
Let us give an example to understand the definition of optimal form.
\begin{example}
Let $I=I(C_5)$ be the edge ideal on the vertex set $\{x_1,x_2,x_3,x_4,x_5\},$ i.e, $I=<x_1x_2,x_2x_3,x_3x_4,x_4x_5,x_5x_1>.$ Consider the monomial $x^{\underline{a}}=x_1^3x_2^4x_3x_4x_5^3.$ Note that we can write $x^{\underline{a}}=x_5^2(x_1x_2)^3(x_2x_3)(x_4x_5)$ and also observe that $b(x^{\underline{a}})$ is maximum in this factorization, so this is an optimal form of the monomial.
\end{example}

\section{Description Of The Graph}

In this section, we try to analyze the properties of the graph $G_{n,r}$ for different values of $n$ and $r.$ We are considering the vertex set $V=\{x_1,\ldots,x_n\}$ in cyclic order to ensure that the vertices $x_1,x_n$ are consecutive. The following figure describe one such graph when $n=7,r=2.$


 \begin{figure}[h!]
 \includegraphics[width=0.2\linewidth]{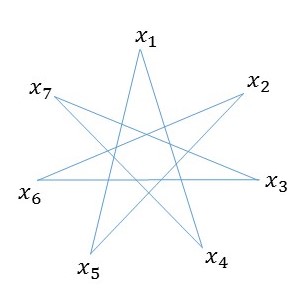}
 \caption{ Example with $n=7$ and $r=2$\label{overflow}}

 \end{figure}

From the description of the graph $G_{n,r},$ it is clear that any set of consecutive $n-(r+1)$ vertices is always a minimal vertex cover. So we might think that $G_{n,r}$ is always unmixed but not all the graph of $G_{n,r}$ are unmixed. Now we characterize the graphs.

\begin{remark}
Note that for $n\leq3$ there are only two graphs $G_{2,0}$ and $G_{3,0},$ one is just an edge and another one is a triangle. So we can assume $n\geq 4.$
\end{remark}

\begin{observation}
There will be no edge between two vertices $\{x_{i_j},x_{i_l}\}$ (where $i_j<i_l$) in the graph $G_{n,r}$  if and only if one of the following conditions hold
\begin{enumerate}\label{twoindependent}
  \item $i_{l}-i_{j}< r+1$

  \item $i_{l}-i_{j}> n-(r+1).$
\end{enumerate}
\end{observation}

\begin{observation}\label{obs1}
The set of vertices $\{x_{i_1},x_{i_2},\ldots,x_{i_k}\}$ is an independent set if and only if there exist a solution for the following equations
\begin{equation}\label{independent}
  {i_p-i_q< r+1}\text{ or  }{i_p-i_q>n-(r+1)}\text{ for }p,q \in\{i_1,i_2,\ldots,i_k\},\text{ with }q<p.
  \end{equation}
\end{observation}

\begin{lem}
The cardinality of a maximal independent set in $G_{n,r}$ is at most $r+1$.
\end{lem}

\begin{proof}
Observe that any set of consecutive $r+1$ vertices always give a maximal independent set. Let us assume that $V^{\prime}=\{x_{i_1},x_{i_2},\ldots,x_{i_k}\}$ is an independent set containing at least two non-consecutive vertices.
We will show that $k\leq r.$ Let $V_{i_j}$ denotes set of vertices which are not connected with $x_{i_j}.$ We choose the vertices $x_{i_1}$ and $x_{i_2}$ such a way that they are not connected. Then we have to choose $x_{i_j}(i_j\geq3)$ from $\displaystyle{\bigcap_{s=1}^{j-1} V_{i_s}}.$ So in this way after choosing $k$ vertices, we show that $\displaystyle{\bigcap_{j=1}^k V_{i_j}}$ will be empty set. Note that $x_{i_1}$ is connected with $n-(2r+1)$ vertices. Since at least two vertices in $V^{\prime}$ are not consecutive so remaining $k-1$ vertices are connected with at least other $k$ vertices apart from $n-(2r+1)$ vertices which are connected with $x_{i_1}.$ Therefore after choosing $k$ independent vertices, the number vertices remain to choose is at most $n-\{(n-2r-1)+k+(k-1)+1\}=2r+1-2k,$ which implies that the maximum value of $k$ is $r.$

\end{proof}

Note that the size of a maximal independent set containing at least two non-consecutive vertices is at most $r$ and the cardinality of a maximal independent set is $r+1$ if and only if all the vertices in that set are consecutive. So we can conclude the following corollary and remark.

\begin{cor}
  Maximal independent sets containing at least two non-consecutive vertices exist if and only if the graph $G_{n,r}$ is not unmixed.
\end{cor}

\begin{remark}
 The independence number of the graph $G_{n,r}$ is $r+1.$
\end{remark}

In the next Lemma we give a class of unmixed graphs arising from $G_{n,r}.$

\begin{lem}\label{unmixed}
The graph $G_{n,r}$ is unmixed for $r=1,2.$
\end{lem}
\begin{proof}
   The proof is clear for $r=1.$ Now consider the case for $r=2.$ If the graph is not unmixed then there exists a maximal independent set containing two non-consecutive vertices. Suppose $\{x_{i_1},x_{i_2}\}$ is an independent set and the vertices are not consecutive, then $i_2-i_1<3$ and $i_2-i_1>1$ which implies $i_2=i_1+2.$ But then the set is not maximal independent set as the set $\{x_{i_1},x_{i_1+1},x_{i_2}\}$ is independent. Therefore the graph is unmixed.
\end{proof}

The next examples shows that $G_{n,r}$ is not always unmixed.
\begin{example}

Consider the graph corresponding to $n=11,r=4.$ Note that the set $\{x_3,x_6,x_7,x_{10}\}$ is a solution of (\ref{independent}), therefore the set is an independent set and also it is maximal independent set. But the graph is not unmixed as $\{x_3,x_4,x_5,x_6,x_7\}$ is also an maximal independent set.

\end{example}

\begin{lem}\label{gapfree}
The graph $G_{n,r}$ is gap-free for $r=1,n\geq 5$ and $r=2,n\geq9.$

\end{lem}

\begin{proof}
  Let $x_ix_j$ and $x_kx_l$ are edges where $i<j$ and $k<l.$ We can assume that $i<l.$ We will show that if $x_ix_k,x_ix_l,x_jx_k$ are not edges then $x_jx_l$ is an edge. Without loss of generality, we can assume that $i=1.$\\
  \textbf{Case 1:} First assume that $j<l.$ Then there are two possibilities, one is $1<k<j<l$ and the other is $1<j<k<l.$\\
  \textbf{Subcase 1:} Let us assume that $1<k<j<l.$ As $l\leq n$ and $j\geq r+2$ (since $x_1x_j$ is an edge), so we have
   \begin{equation}\label{eq1}
   l-j\leq n-(r+2)\leq n-(r+1).
    \end{equation}
   Now $x_1x_j$ is edge and $x_1x_k$ not an edge imply that $k\leq r+1.$ Also $x_1x_j$ edge and $x_1x_l$ not edge imply that \begin{equation}\label{eq*}
       l-1>n-(r+1) \text{ i.e. }
       l\geq n-r+1.
      \end{equation}

    Since $x_jx_k$ is not an edge, we have either $j-k<r+1$ or $j-k>n-(r+1).$ But note that $j-k>n-(r+1)$ is not possible for $r=1,2.$ Suppose possible then $j-k>n-3,$ which implies that $j\geq k+n-2,$ also $k\geq 2,$ hence $j\geq n$ which contradicts the fact that $j<l\leq n.$ So we need to prove only for $j-k<r+1.$ Then $j\leq k+r\leq 2r+1,$ which together with (\ref{eq*}) give $l-j\geq n-3r.$ Observe that for $r=1,n\geq 5$ and $r=2,n\geq9$ the inequality $l-j\geq r+1$ satisfied, hence by (\ref{eq1}) $x_jx_l$ is an edge. \\

  \textbf{Subcase 2:} Now consider the case $1<j<k<l$ and also $x_1x_j,$ $x_kx_l$ are edges. Assume that $x_1x_k$ is not an edge then $k\geq n-r+1,$ which implies that $l-k\leq n-(n-r+1)=r-1.$ Hence we arrived at a contradiction that $x_kx_l$ is an edge. Then $x_1x_k$ is an edge.\\

  \textbf{Case 2:} If $l<j,$ then $1<k<l<j.$ Since $r\leq 2,$ if $x_1x_l$ is not edge then $k$ has to be $2$ and $l$ has to be $3,$ but this contradicts that $x_kx_l$ is an edge. Thus $x_1x_l$ is an edge.

\end{proof}

\section{Regularity comparison}
The main aim of this section is to compare the regularities of the ordinary powers and symbolic powers of $I=I(G_{n,r}).$ Using the local cohomology technique we show that Minh's Conjecture is true for the graph $G_{n,r}.$

\begin{thm}\label{IT}
Let $I=I(G_{n,r})$ be the edge ideal for the graph $G_{n,r},$ then $I^t=(L(t)).$
\end{thm}
\begin{proof}
Observe that any set of consecutive $n-(r+1)$ vertices forms a minimal vertex cover for $G_{n,r}$. As the relation $I^t\subseteq (L(t))$ always holds, we need to show that $(L(t))\subseteq I^t.$
We divide the proof into 3 cases depending on the number of ancillaries present in the optimal form of the monomial.  Observe that there can be at most $(r+1)$ ancillaries in the optimal form of a monomial.\\

\textbf{Case 1:} Let $x^{\underline{a}}\in (L(t)).$ Suppose there is no ancillary or only one ancillary with degree 1 is present, then we can write $x^{\underline{a}}=x_i^{c_i}p(x),$ where $c_i\leq 1$ and $\deg(p(x))\geq 2t-1.$ Note that $p(x)$ can be factored into edges only, which implies that $x^{\underline{a}}\in I^t$ and hence in this case $(L(t))\subseteq I^t.$  \\

We prove case 2 and 3 by contrapositive. Let $x^{\underline{a}}=x_1^{a_1}\cdots x_n^{a_n}\notin I^t$ which implies that $b(x^{\underline{a}})< t,$ now if we can find a minimal vertex cover $V^{\prime}$ such that $W_{V^{\prime}}(x^{\underline{a}})\leq b(x^{\underline{a}})$, then $x^{\underline{a}}\notin (L(t)).$

\textbf{Case 2:} Assume that $i$ ancillaries are present in the optimal form, where $2\leq i\leq r.$ Let $x_1^{c_1},\cdots,x_i^{c_i}$ are the ancillaries. Consider the following set of minimal vertex covers  $ L=\{\{x_{i+1},\cdots, x_{n-r+i-1}\},\cdots, \\ \{x_{r+2},\ldots,x_{n}\}\}.$ The aim here is to find a minimal vertex cover $V^{\prime}$ from the set $L$ such that $W_{V^{\prime}}(x^{\underline{a}})\leq b(x^{\underline{a}}).$ Our claim is that we can find a minimal vertex cover $V^{\prime}$ from the set $L$ such that there will be no edge between the vertices of $V^{\prime}$  in the optimal form of the monomial and if so then $V^{\prime}$ will be the suitable candidate for minimal vertex cover. Thus it is enough to prove that if at least one edge is present from each vertex set of $L$ then the monomial will not be optimal form.
Without loss of generality, we can write the monomial in the following way
$x^{\underline{a}}=x_1^{c_1}{x_2}^{c_2}\cdots {x_i}^{c_i}(x_{k_{(i+1)1}}x_{k_{(i+1)2}})(x_{k_{(i+2)1}}x_{k_{(i+2)2}})\cdots (x_{k_{(r+2)1}}x_{k_{(r+2)2}})$ and assume that it is the optimal form, where $x_1^{c_1},{x_2}^{c_2},\ldots, {x_i}^{c_i}$ are the ancillaries and $(x_{k_{(i+1)1}}x_{k_{(i+1)2}}),\cdots, (x_{k_{(r+2)1}} x_{k_{(r+2)2}})$ are the edges between the vertices of $ \{x_{i+1},\cdots, x_{n-r+i-1}\},\cdots, \{x_{r+2},\ldots,x_{n}\}$ respectively. If two edges belong to the same vertex set of $L$ then we will keep their positions according to the sum of their indices. For example if $(x_{k_{(i+1)1}}x_{k_{(i+1)2}}),(x_{k_{(i+2)1}}x_{k_{(i+2)2}})$ are the edges from the vertex set $\{x_{i+1},\cdots, x_{n-r+i-1}\}$ and suppose that $k_{(i+1)1}+k_{(i+1)2}\leq k_{(i+2)1}+k_{(i+2)2}$ then we keep the edge $(x_{k_{(i+1)1}}x_{k_{(i+1)2}})$ before $(x_{k_{(i+2)1}}x_{k_{(i+2)2}}).$ So by rearranging we can assume that if they belong to same the vertex set then $k_{(i+1)1}+k_{(i+1)2}\leq k_{(i+2)1}+k_{(i+2)2}.$ From Observation \ref{twoindependent}, it follows that $x_ix_j(i<j)$ is an edge if and only if $r+1\leq j-i\leq n-(r+1).$ Therefore the following equations be satisfied


\begin{align} \label{eq:special}
i+1\leq k_{(i+1)1}<k_{(i+1)2}\leq n-r+i-1 &\text{ and }  r+1\leq k_{(i+1)2}-k_{(i+1)1}\leq n-(r+1)\nonumber \\
  i+2\leq k_{(i+2)1}<k_{(i+2)2}\leq n-r+i & \text{ and } r+1\leq k_{(i+2)2}-k_{(i+2)1}\leq n-(r+1)\\ \nonumber \vdots  \\
r+2\leq k_{(r+2)1}<k_{(r+2)2}\leq n &\text{ and } r+1\leq k_{(r+2)2}-k_{(r+2)1}\leq n-(r+1)\nonumber.
\end{align}

Now
\begin{align*}
x^{\underline{a}}&= x_1^{c_1}{x_2}^{c_2}\cdots {x_i}^{c_i}(x_{k_{(i+1)1}}x_{k_{(i+1)2}})(x_{k_{(i+2)1}}x_{k_{(i+2)2}})\cdots (x_{k_{(r+2)1}}x_{k_{(r+2)2}})\\
&= x_1^{c_1-1}{x_2}^{c_2}\cdots {x_{i-1}}^{c_{i-1}}{x_i}^{c_i-1}x_1x_i(x_{k_{(i+1)1}}x_{k_{(i+1)2}})(x_{k_{(i+2)1}}x_{k_{(i+2)2}})\cdots (x_{k_{(r+2)1}}x_{k_{(r+2)2}})\\
&=x_1^{c_1-1}{x_2}^{c_2}\cdots {x_{i-1}}^{c_{i-1}}{x_i}^{c_i-1}(x_1x_{k_{(r+2)1}})(x_ix_{k_{(i+1)2}})(x_{k_{(i+1)1}}x_{k_{(i+2)2}})\cdots (x_{k_{(r+1)1}}x_{k_{(r+2)2}}).
\end{align*}

It is enough to show that all the pairs of the last expression of $x^{\underline{a}}$ are edges.
From the equations(\ref{eq:special}), it follows that $r+1\leq k_{(r+2)1}-1$ and $k_{(r+2)1}-1\leq k_{(r+2)2}-(r+1)\leq n-(r+1).$ Hence $r+1\leq k_{(r+2),1}-1\leq n-(r+1).$
Therefore $x_1x_{k_{(r+2)1}}$ is an edge.
Again from the equations (\ref{eq:special}), we can write $k_{(i+1)2}\geq k_{(i+1)1}+(r+1)\geq (i+1)+(r+1),$ which implies that
 \begin{equation}\label{(i+1)2}
   k_{(i+1)2}-i\geq r+1.
 \end{equation}
Also note that $k_{(i+1)2}-i\leq n-(r+1).$ Therefore $r+1\leq k_{(i+1)2}-i\leq n-{r+1},$ hence $x_ix_{k_{(i+1)2}}$ is an edge. Next we show that $x_{k_{(i+1)1}}x_{k_{(i+2)2}}$ is an edge. First assume that $(x_{k_{(i+1)1}}x_{k_{(i+1)2}}),\\ (x_{k_{(i+2)1}}x_{k_{(i+2)2}})$ are the edges from the vertex set $\{x_{i+1},\cdots, x_{n-r+i-1}\}$ then $k_{(i+1)1}+k_{(i+1)2}\leq k_{(i+2)1}+k_{(i+2)2}.$ Then either $k_{(i+1)1}\leq k_{(i+2)1}$ or $k_{(i+1)2}\leq k_{(i+2)2}$  hold which implies that $k_{(i+2)2}-k_{(i+1)1}\geq k_{(i+2)2}-k_{(i+2)1}\geq r+1$ or $k_{(i+2)2}-k_{(i+1)1}\geq k_{(i+1)2}-k_{(i+1)1}\geq r+1.$ Also observe that the maximum difference between the indices from the vertex set of $L$ is $n-r-2$ hence $(x_{k_{(i+1)1}}x_{k_{(i+2)2}})$ is an edge. The remaining case is that if they do not belong to the same vertex set, in that case $k_{(i+2)2}=n-r+i,$ then it is obvious that $(x_{k_{(i+1)1}}x_{k_{(i+2)2}})$ is an edge as $k_{(i+2)2}-k_{(i+1)1}\leq n-(r+1).$ In the same procedure, we can show that other pairs are also edges. Then $x^{\underline{a}}$ is not in optimal form.\\

If only one ancillary is present, then $i=1.$ Assume that $x_1^{c_1}$ is the ancillary with $c_1\geq 2.$ Then take $x_i=x_1$ and follow the process described above. In this case we have to show only that $x_1x_{k_{(i+1)2}}$ is an edge. From equation (\ref{(i+1)2}) it follows that  $k_{(i+1)2}-1\geq r+1.$ By putting $i=1$ in equations (\ref{eq:special}) we get that $k_{(i+1)2}\leq n-r,$ i.e. $k_{(i+1)2}-1\leq n-(r+1).$ Therefore $x_1x_{k_{(i+1)2}}$ is an edge. Hence we can find a minimal vertex cover $V^{\prime}$ from the set $L$ such that no edges of $V^{\prime}$ will be present in the optimal form of $x^{\underline{a}},$ so $W_{V^{\prime}}(x^{\underline{a}})\leq b(x^{\underline{a}})<t.$\\

\textbf{Case 3:} Let us assume that $(r+1)$ ancillaries are present in the optimal form and $x_1^{c_1},\ldots,x_{r+1}^{c_{r+1}}$ are the ancillaries. Then we can write the optimal form of the monomial in the following way, $x^{\underline{a}}=x_1^{c_1}\cdots x_{r+1}^{c_{r+1}}e_{1(r+2)}^{b_{1(r+2)}}\cdots e_{1(n-r)}^{b_{1(n-r)}}e_{2(r+3)}^{b_{2(r+3)}}\cdots e_{2(n-r+1)}^{b_{2(n-r+1)}}\cdots  e_{(r+1)(2r+2)}^{b_{(r+1)(2r+2)}}\cdots e_{(r+1)n}^{b_{(r+1)n}}.$ Consider the vertex set $V^{\prime}=\{x_{r+2},\ldots,x_n\}.$ Clearly $V^{\prime}$ is a minimal vertex cover. Observe that in the optimal form of $x^{\underline{a}}$ there is no edge between the vertices of ${V^{\prime}}.$ Thus $W_{V^{\prime}}(x^{\underline{a}})=a_{r+2}+\cdots+a_n$ and $b(x^{\underline{a}})=b_{1(r+2)}+\cdots+b_{1(n-r)}+b_{2(r+3)}+\cdots+b_{2(n-r+1)}+\cdots+b_{(r+1)(2r+2)} + \cdots+ b_{(r+1)n}.$ Now $b_{1(r+2)}=a_{r+2},b_{1(r+3)}+b_{2(r+3)}=a_{r+3}, b_{1(r+4)}+b_{2(r+4)}+b_{3(r+4)}=a_{(r+4)}$,\ldots,$b_{1(n-r)}+b_{2(n-r)}+\cdots+b_{(n-2r-1)(n-r)}=a_{n-r},\cdots,b_{(r+1)n}=a_n.$ Therefore $W_{V^{\prime}}(x^{\underline{a}})=b(x^{\underline{a}})< t.$

\end{proof}


 \begin{lem} Let $H_{n,r}$ be the class of unmixed graphs of the form $G_{n,r}$ and let $I=I(G)$ be the edge ideal of some graph in $H_{n,r}.$ Then we have

\begin{equation*}
\begin{split}
D(t) = \{x_1^{a_1}x_2^{a_2}\cdots x_n^{a_n}& \mid a_{i_1}+a_{i_2}+\cdots+a_{i_{n-(r+1)}}\geq t \mbox{ for consecutive $n-(r+1)$}\mbox{ tuple with } \\ &  \{i_1,i_2,\ldots,i_{n-(r+1)}\}\subseteq\{1,2,\ldots,n\}
              \mbox{ and } a_1+\cdots +a_n\leq 2t-1\}
\end{split}
\end{equation*}
\end{lem}
\begin{proof}
  The set of all consecutive $n-(r+1)$ vertices forms the set of minimal vertex cover for $H_{n,r}.$ Let $x^{\underline{a}}={x_1^{a_1}x_2^{a_2}\cdots x_n^{a_n}\in D(t) }$ then $a_1+\dots +a_n\leq 2t-1$ and since $x^{\underline{a}}\in I^{(t)},$ so for any minimal vertex cover $V,$ $W_V(x^{\underline{a}})\geq t$ which implies that  $a_{i_1}+a_{i_2}+\dots+a_{i_{n-(r+1)}}\geq t \mbox{ for all consecutive $n-(r+1)$} \\ \mbox{tuple with }  \{i_1,i_2,\dots,i_{n-(r+1)}\}\subseteq\{1,2,\dots,n\}$. Hence the result follows.
\end{proof}

 \begin{lem}\label{containment}
Let $I=I(G_{n,r})$ be the edge ideal of a graph $G_{n,r},$ and let $\m$ be the homogeneous maximal ideal of $R.$ Then ${\m}^{t}I^{(t)}\subseteq I^t.$
\end{lem}
\begin{proof}

Let $x^{\underline{a}}\in {\m}^t$ and $x^{\underline{b}}\in {I}^{(t)},$ then $\deg(x^{\underline{a}})\geq t$ as well as $\deg(x^{\underline{b}})\geq t.$ Therefore $\deg(x^{\underline{a}+\underline{b}})\geq 2t,$ which implies that $x^{\underline{a}+\underline{b}}\in (L(t)),$ hence by Theorem \ref{IT}, $x^{\underline{a}+\underline{b}}\in I^t.$ Thus the result follows.

\end{proof}

\begin{thm}\label{reginequality}
  Let $I=I(G_{n,r})$ be the edge ideal for the graph $G_{n,r}$. Then for all $t\geq 1$ we have $$\reg I^{(t)}= \reg I^t.$$
 \end{thm}
 \begin{proof}

From Lemma \ref{containment}, it follows that $\dim {I^{(t)}}/{I^t}=0,$ therefore $H_{\m }^i(I^{(t)}/I^t)=0$ for $i>0.$ Now consider the following short exact sequence
   \begin{equation}\label{exact1}
   0\rightarrow I^{(t)}/I^t \rightarrow R/I^t \rightarrow R/I^{(t)}\rightarrow 0.
   \end{equation}

   Applying local cohomology functor we get $H_{\m}^i(R/I^{(t)})\cong H_{\m}^i(R/I^{t})$ for $i\geq 1$ and the following short exact sequence
   \begin{equation} \label{exact2}
0\rightarrow H_{\m}^0(I^{(t)}/I^t) \rightarrow H_{\m}^0(R/I^t) \rightarrow H_{\m}^0(R/I^{(t)})\rightarrow 0.
   \end{equation}

   From $(\ref{exact2}),$ it follows that $a_0(R/I^{(t)})\leq a_0(R/I^t).$
   So we can conclude that $\reg R/I^{(t)}=\max\{a_i(R/I^{(t)})\\+i~|~i\geq 0\}\leq\max\{a_i(R/I^{t})+i~|~i\geq 0\}=\reg(R/I^t). $ Therefore $\reg I^{(t)}\leq \reg I^t.$ Now we prove the other inequality.


  Since $\dim(I^{(t)}/I^t)=0,$ it follows that $\reg I^{(t)}/I^t=a_0(I^{(t)}/I^t).$ If a monomial $x^{\underline{a}}\in I^{(t)}$ with $\deg(x^{\underline{a}})\geq 2t,$ then $x^{\underline{a}}\in (L(t))$ and hence by Theorem \ref{IT}, it follows that $x^{\underline{a}}\in I^t.$ Thus $a_0(I^{(t)}/I^t)\leq 2t-1,$ which implies $\reg(I^{(t)}/I^t)\leq 2t-1.$ Since $I^t \subseteq I ^{(t)},$ we have $2t\leq \reg I^{(t)}.$ Therefore $\reg I^{(t)}/I^t\leq \reg R/I^{(t)}.$
  Now by applying \cite[Corollary 20.19]{eisenbud2013commutative} to (\ref{exact1}) we get $\reg R/I^t\leq \max\{\reg I^{(t)}/I^t,\reg R/I^{(t)}\},$ which implies $\reg I^t\leq \reg I^{(t)}$ and hence $\reg I^t = \reg I^{(t)}.$

\end{proof}

\begin{remark}
If $r=\Bigl\lfloor\dfrac{n}{2}\Bigr\rfloor -1$ and $n$ is odd then $G_{n,r}$ gives the class of odd cycles which has been studied by Gu et al. in \cite{gu2018symbolic}. If $r=0$ then $G_{n,r}$ gives the class of complete graph which we have studied in \cite{chakraborty2019invariants}. Both of the above results follow as a special case of our result.

\end{remark}

\section{Invariants of Edge Ideals}

In this section, we calculate Waldschmidt constant and resurgence for the edge ideal $I=I(G_{n,r})$ of the whole class.

\begin{defn}
 An automorphism of a graph $G$ is a bijection $\pi : V (G) \rightarrow V (G)$ with the property that $e$ is an edge of $G$ if and only if $\pi(e)$ is an edge as well.
\end{defn}

\begin{defn}
A graph $G$ is vertex-transitive if for all $u,v \in V(G)$ there is an automorphism $\pi$ of $G$ with $\pi(u)=v.$
\end{defn}

\begin{lem}
The graph $G_{n,r}$ is vertex-transitive graph.
\end{lem}
\begin{proof}
For every vertex $x_i,x_j \in V(G)$ we have to produce an automorphism $\pi$ of $G$ such that $\pi(x_i)=x_j.$ Let $\overline{k}\cong j-i+k$ mod $n,$ where $x_0=x_n.$ Now consider the map $\pi(x_k)=x_{\overline{k}}.$ Observe that it is an automorphism.
\end{proof}

\begin{proposition}\label{alpha}
Let $I=I(G_{n,r})$ be the edge ideal of a graph $G_{n,r}$. Then $\alpha(I^{(t)}) < \alpha (I^s)$ if and only if $I^{(t)}\nsubseteq I^s.$
\end{proposition}
\begin{proof}
By Theorem \ref{IT}, we have $(L(t))=I^t.$ Then the proof follows similarly from \cite[Lemma 5.5]{janssen2017comparing}.
\end{proof}
\bigskip
 Let us recall some notions used in section $5$ of the article \cite{janssen2017comparing}. Note that the graph $G_{n,r}$ contains unmixed graphs as well as non-unmixed graphs. For unmixed graphs, there are only $n$ minimal vertex covers and they are precisely consecutive $n-(r+1)$ vertices. But for non-unmixed graph there are some extra minimal vertex covers and they contain more than $n-(r+1)$ vertices. Now fix the minimal vertex covers $V_1,\ldots,V_k$ and order them so that $|V_i| \leq |V_{i+1}|.$ We define the minimal vertex cover matrix $A = (a_{ij})$ to be the matrix
of $0$'s and $1$'s defined by:
\begin{align*}
                                    a_{ij} & =  \left\{\begin{array}{ll}
                                     0 & \text{ if $x_j\notin V_i $} \\
                                       1 & \text{  if $x_j\in V_i$ }\\
                                        \end{array} \right.\\
                                        \end{align*}
We first seek a lower bound of $\alpha(I^{(t)})$ using linear programming.
Consider the following linear program, where A is the minimal vertex cover
matrix,\\
$$b=\begin{bmatrix}
1\\
\vdots\\

1\\
\end{bmatrix}
\text{  and  } c=\begin{bmatrix}
t\\
\vdots\\

t\\
\end{bmatrix}$$

\begin{equation}\label{1}
\begin{aligned}
& {\text{minimize}}
& & b^Ty \\
& \text{subject to}
& & Ay \geq c \; \text{and } y\geq 0.
\end{aligned}
\end{equation}

Observe that if $y^*$ is the value which realizes (\ref{1}), we have $\alpha(I^{(t)})\geq b^Ty^*.$

Now consider the submatrix $A^{\prime}$ of $A$ consisting of first $n$ rows of $A,$ then it is a $n\times n$ matrix. Thus we create a sub-program of (\ref{1}).
\begin{equation}\label{3}
\begin{aligned}
& {\text{minimize}}
& & b^Ty \\
& \text{subject to}
& & A^{\prime}y \geq c \; \text{and } y\geq 0.
\end{aligned}
\end{equation}

\begin{lem}\label{alphain}

 Let $I=I(G_{n,r})$ be the edge ideal of the graph $G_{n,r},$ then $\alpha(I^{(t)})\geq\Bigl\lceil\dfrac{nt}{n-(r+1)}\Bigr\rceil.$
\end{lem}

\begin{proof}
It is easy to check that $$y^*=\begin{bmatrix}
\frac{t}{n-(r+1)}\\
\vdots\\

\frac{t}{n-(r+1)}\\
\end{bmatrix}$$
 is a feasible solution of (\ref{3}). In this case $b^Ty^*=\frac{nt}{n-(r+1)}.$ To show that this is the value of (\ref{3}), we make use of the fundamental theorem of linear programming by showing the existence of an $x^*$ which produces the same value for the dual linear program:

\begin{equation}\label{2}
\begin{aligned}
& {\text{maximize}}
& & c^Tx \\
& \text{subject to}
& & {A^{\prime}}^Tx \leq b \; \text{and } x\geq 0.
\end{aligned}
\end{equation}

Let $$x^*=\begin{bmatrix}
\frac{1}{n-(r+1)} \\
\vdots\\

\frac{1}{n-(r+1)}\\
\end{bmatrix}.$$
 It is straightforward to check that $x^*$ is a feasible solution of the dual problem and $c^Tx^*=b^Ty^*=\frac{nt}{n-(r+1)}.$

Observe that (\ref{1}) is obtained from (\ref{3}) by introducing additional constraints  and also $\alpha(I^{(t)})$ is an integer. Hence the result follows.
\end{proof}

In order to compute Waldschmidt constant we now introduce the fractional chromatic number.
\begin{defn}\cite{scheinerman2011fractional}
A $b$-fold coloring of a graph $G$ assigns to each vertex of $G$ a set of $b$ colors so that adjacent vertices receive disjoint sets of colors. We say that $G$ is $a:b$-colorable if it has a $b$-fold coloring in which the colors are drawn from a palette of $a$ colors. We sometimes refer to such a coloring as an $a:b$-coloring. The least $a$ for which $G$ has a $b$-fold coloring is the $b$-fold chromatic number of $G$, denoted $\chi_b(G).$
Since $\chi_{a+b} (G) \leq  \chi_{ a }(G) + \chi_{b}(G),$ we define the fractional chromatic number to be $ \chi^{*}(G) = \displaystyle{\lim _{b\rightarrow \infty}\frac{\chi_{b}(G)}{b}=\inf_{b}\frac{\chi_{b}(G)}{b}} .$

\end{defn}

\begin{thm}
Let $I=I(G_{n,r})$ be the edge ideal of a graph $G_{n,r}$ then the Waldschmidt constant $\widehat{\alpha}(I)=\frac{n}{n-(r+1)}.$
\end{thm}

\begin{proof}
By \cite[Proposition 3.1.1]{scheinerman2011fractional}, it follows that the fractional chromatic number $\chi^*(G_{n,r})=\frac{n}{r+1},$ as independence number is $r+1.$ Thus by \cite[Theorem 4.6]{bocci2016} we have $\widehat{\alpha}(I))=\frac{n}{n-(r+1)}.$
\end{proof}

\begin{thm}
Let $I=I(G_{n,r})$ be the edge ideal then $\rho(I)=\frac{2(n-(r+1))}{n}.$

\end{thm}

  \begin{proof}

Consider the set $T=\{\frac{s}{t}\mid I^{(s)}\nsubseteq I^t \}.$ For any $\frac{s}{t}\in T$ by Proposition \ref{alpha}, we have $\alpha(I^{(s)})< \alpha(I^t).$
Then by applying Lemma \ref{alphain}, we can write
$\frac{ns}{n-(r+1)}<2t, $ which implies $\frac{s}{t}<\frac{2(n-(r+1))}{n}.$ Hence $\rho(I)\leq\frac{2(n-(r+1))}{n}.$ As $\frac{\alpha(I)}{\widehat{\alpha}(I)}=\frac{2(n-(r+1))}{n},$
 by \cite[Theorem 1.2]{guardo2013}, we have ${\alpha(I)}/{\hat{\alpha}(I)}\leq \rho(I),$ therefore $\frac{2(n-(r+1))}{n}\leq \rho(I)$ and hence $\rho(I)=\frac{2(n-(r+1))}{n}.$

        \end{proof}

 \end{document}